\documentclass[amsart, 10pt]{article}

\usepackage{amsmath, amsthm, amsfonts, amssymb}
\usepackage{eucal}
\usepackage{dsfont,latexsym}

\usepackage{graphicx}
\usepackage{color}

\renewcommand{\H}{\mathds{H}}
\newcommand{\N}{{\mathds N}}
\newcommand{\R}{{\mathds R}}

\newcommand{\B}{\mathcal{B}}

\renewcommand{\div}{\operatorname{div}}

\newtheorem{theorem}{Theorem}
\newtheorem{lemma}[theorem]{Lemma}
\newtheorem{proposition}[theorem]{Proposition}
\newtheorem{corollary}[theorem]{Corollary}
\newtheorem{definition}[theorem]{Definition}

\begin{document}

\title{A Lewy-Stampacchia Estimate\\
for quasilinear variational inequalities\\
in the Heisenberg group}

{\author{Andrea Pinamonti \& Enrico Valdinoci}}

\maketitle

\begin{abstract}
We consider an obstacle problem in the
Heisenberg group framework, and we prove that the
operator on the obstacle
bounds pointwise the operator on the solution.
More explicitly, if~$\varepsilon\ge0$ and~$\bar u_\varepsilon$ minimizes
the functional
$$ \int_\Omega(\varepsilon+|\nabla_{\H^n}u|^2)^{p/2}$$
among the functions with prescribed Dirichlet boundary
condition that stay below a smooth obstacle~$\psi$, then
\begin{eqnarray*} && 0\le
\div_{\H^n}\, \Big( (\varepsilon+|\nabla_{\H^n}\bar u_\varepsilon|^2)^{(p/2)-1}
\nabla_{\H^n}\bar u_\varepsilon\Big)
\\ &&\qquad\le \left(
\div_{\H^n}\, \Big( (\varepsilon+|\nabla_{\H^n}\psi|^2)^{(p/2)-1}
\nabla_{\H^n}\psi\Big)
\right)^+.\end{eqnarray*}
\end{abstract}

\section*{Introduction}

In this paper, we extend the so called dual estimate
of~\cite{LewySt} 
to the obstacle problem
for quasilinear elliptic equations
in the Heisenberg group.

The notation we use is the standard one:
for~$n\ge1$, we consider~$\R^{2n+1}$
endowed with the group law
\begin{eqnarray*} && \big( x^{(1)},y^{(1)},t^{(1)}
\big)\circ\big(x^{(2)},y^{(2)},t^{(2)}\big)
\\ && \qquad:=\Big(
x^{(1)}+x^{(2)},y^{(1)}+y^{(2)},
t^{(1)}+t^{(2)}+2(x^{(2)}\cdot y^{(1)}-
x^{(1)}\cdot y^{(2)} ) \Big),\end{eqnarray*}
for any~$(x^{(1)},y^{(1)},t^{(1)})$, 
$(x^{(2)},y^{(2)},t^{(2)})\in\R^n\times\R^n\times\R$,
where the ``$\cdot$'' is the standard Euclidean scalar product.

Then, we
denote by~$\H^n$ the~$n$-dimensional Heisenberg 
group, i.e.,~$\R^{2n+1}$ endowed with this group law.

The coordinates are usually written as~$(x,y,t)\in\R^n\times
\R^n\times\R$, and, as
customary, we introduce the {left invariant} vector fields~$(X,Y)$
induced by the group law
$$X_j:=\frac{\partial}{\partial x_j} +2y_j\frac{\partial}{\partial t}
{\mbox{
and
}}
Y_j:=\frac{\partial}{\partial y_j} -2x_j\frac{\partial}{\partial t},$$
for~$j=1,\dots,n$,
and the horizontal gradient~$\nabla_{\H^n}:=(X,Y)$.
The main issue of the 
Heisenberg
group is that~$X$ and $Y$ do not commute, that is
\[[X,Y]=4\frac{\partial}{\partial t}\neq 0.\]
We are interested in studying the obstacle problem
in this framework.
For this, we consider a smooth
function~$\psi:\H^n\rightarrow\R$,
which will be our obstacle (more
precisely,~$\psi$ is supposed to
have continuous derivatives of second order in~$X$ and~$Y$).

Fixed a bounded open set~$\Omega$ with smooth boundary,
and~$p\in (1,+\infty)$,
we consider the space~$W^{1,p}_{\H^n}(\Omega)$ to
be the set of all functions~$u$ in~$L^p(\Omega)$
whose distributional horizontal derivatives~$X_ju$ and~$Y_ju$
belong to~$L^p(\Omega)$, for~$j=1,\dots,n$.

Such space is naturally endowed with the norm
$$ \|u\|_{ W^{1,p}_{\H^n}(\Omega) }:=
\|u\|_{L^p(\Omega)}+
\sum_{j=1}^n \Big( \|
X_j u\|_{L^p(\Omega)}+\|Y_ju\|_{L^p(\Omega)}\Big).$$
We call~$W^{1,p}_{\H^n,0}(\Omega)$ the closure of~$C^\infty_0(\Omega)$
with respect to this norm.

We fix
a smooth domain~$\Omega_\star\Supset {\Omega}$, $u_\star\in
W^{1,p}_{\H^n}(\Omega_\star)\cap L^\infty(\Omega_\star)$
and we introduce 
the space
$$ {\mathcal{K}}:= \big\{ u\in W^{1,p}_{\H^n}(\Omega)
{\mbox{ s.t. }} u\le \psi, {\mbox{ and }} u-u_\star\in
W^{1,p}_{\H^n,0}(\Omega) \big\}.$$
Loosely speaking,~${\mathcal{K}}$ is the space of
all the functions having prescribed Dirichlet
boundary datum equal to~$u_\star$
along~$\partial\Omega$ and that stay below the obstacle~$\psi$.

Now we consider a parameter~$\varepsilon\ge0$
and we deal with
the variational problem
\begin{equation}\label{PB}
\inf_{u\in{\mathcal{K}}} {\mathcal{F}}_\varepsilon (u;\Omega),
{\mbox{ where }}
{\mathcal{F}}_\varepsilon (u;\Omega):=
\int_\Omega(\varepsilon+|\nabla_{\H^n}u|^2
)^{p/2}.\end{equation}
By direct methods, it is seen that such infimum is attained
(see, e.g., the compactness result
in~\cite{VSC, Dan}
or references therein)
and so we consider a minimizer~$\bar u_\varepsilon$.

Then,~$\bar u_\varepsilon$ is a solution of the variational
inequality\footnote{Formula~\eqref{ines} may be easily obtained
this way. Fixed~$v\in W^{1,p}_{\H^n}(\Omega)$ with~$v
\le \psi$, and~$v-\bar u_\varepsilon\in
W^{1,p}_{\H^n,0}(\Omega) $,
for any~$t\ge0$, 
let~$u^{(t)}:=\bar u_\varepsilon+t(v-\bar u_\varepsilon)$.
Notice that
$$u^{(t)}:=(1-t)\bar u_\varepsilon+tv\le(1-t)\psi+t\psi\le\psi,$$ 
hence~$u^{(t)}\in {\mathcal{K}}$.
So, by the minimality of~$\bar u_\varepsilon$, we 
have~${\mathcal{F}}_\varepsilon (u^{(0)};\Omega)
={\mathcal{F}}_\varepsilon (\bar u_\varepsilon;\Omega)
\le{\mathcal{F}}_\varepsilon (u^{(t)};\Omega)$ for any~$t\ge0$.
Consequently,
\begin{eqnarray*} 0&\le& \lim_{t\searrow0}\frac{{\mathcal{F}}_\varepsilon 
(u^{(t)};\Omega)-{\mathcal{F}}_\varepsilon (u^{(0)};\Omega)}{t}
\\ &=&
\int_\Omega (\varepsilon+|\nabla_{\H^n}\bar u_\varepsilon|^2
)^{(p-2)/2} \nabla_{\H^n}
\bar u_\varepsilon\cdot \nabla_{\H^n}(v-\bar u_\varepsilon),
\end{eqnarray*}
that is~\eqref{ines}.}
\begin{equation}\label{ines}
\int_\Omega (\varepsilon+|\nabla_{\H^n}\bar u_\varepsilon|^2
)^{(p-2)/2} \nabla_{\H^n}
\bar u_\varepsilon\cdot \nabla_{\H^n}(v-\bar u_\varepsilon)\,
\ge\,0,\end{equation}
for any~$v\in W^{1,p}_{\H^n}(\Omega)$ with~$v
\le \psi$, and~$v-\bar u_\varepsilon\in
W^{1,p}_{\H^n,0}(\Omega) $.

These kind of variational inequalities
has now receiving a considerable attention (see, e.g.,~\cite{Dan2}
and references therein), even when $p=2$ (notice
that in such a 
case~$\varepsilon$
does not play any role).
We observe that, when~$p\ne2$, the operator
driving the variational inequality in~\eqref{ines}
is not linear anymore (in fact, it may be seen
as the Heisenberg group version of the $p$-Laplace
operator): for these kind of operators
even the regularity theory is more
problematic than expected at a first glance,
and many basic fundamental questions
are still open (see, e.g.,~\cite{Dom}, \cite{MinMan},
\cite{Zhong1} and~\cite{Zhong2}):
this is a crucial difference with respect to
the Euclidean case,
so we think it is worth dealing with the problem
in such a generality.

Now, we introduce the set of~$p$'s for
which our main result holds.
The definition we give is slightly technical,
but, roughly speaking, consists in taking
the set of all the $p$'s for
which a pointwise bound for the operator
of a sequence of minimal solutions
is stable under uniform limit.
The further difficulty of taking this
assumption is due to the lack
of a thoroughgoing regularity theory for
the quasilinear Heisenberg equation
(as remarked in
Lemma~\ref{p2P} at the end of this
paper, this technicality may be skipped when~$p=2$).

\begin{definition}\label{D} Let~$p\in(1,+\infty)$.
We say that~$p\in {\mathcal{P}}(\psi,\Omega)$ 
if the following
property holds true:
 
For any~$\varepsilon>0$, any~$v\in W^{1,p}_{\H^n}(\Omega)$,
any~$M>0$, any sequence~$F_k=F_k(r,\xi)\in 
C([-M,M]\times\Omega)$, 
with~$F_k(\cdot,\xi)\in 
C^1([-M,M])$ and
\begin{equation}\label{1C4}\begin{split}
& 0\le \partial_r F_k\le 
\left(
\div_{\H^n}\, \Big( (\varepsilon+|\nabla_{\H^n}\psi|^2)^{(p/2)-1}
\nabla_{\H^n}\psi\Big)
\right)^+ ,
\end{split}\end{equation}
if~$u_k:\Omega\rightarrow [-M,M]$ is a sequence
of minimizers of the functional
\begin{equation}\label{905} \int_\Omega
\frac{1}{p}(\varepsilon+|\nabla_{\H^n}
u(\xi)|^2)^{p/2}+F_k(u(\xi),\xi)\,d\xi\end{equation}
over the functions~$u\in W^{1,p}_{\H^n}(\Omega)$, $u-v\in
W^{1,p}_{\H^n,0}(\Omega)$,
with the property that~$u_k$ converges to some~$u_\infty$
uniformly in~$\Omega$, we have that
\begin{equation}\label{TX}\begin{split} & 0\le
\div_{\H^n}\, \Big( 
(\varepsilon+|\nabla_{\H^n}u_\infty|^2)^{(p/2)-1}
\nabla_{\H^n}u_\infty\Big)
\\ &\qquad
\le
\left(
\div_{\H^n}\, \Big( (\varepsilon+|\nabla_{\H^n}\psi|^2)^{(p/2)-1}
\nabla_{\H^n}\psi\Big)
\right)^+.\end{split}\end{equation}
\end{definition}

As remarked\footnote{As
usual, the superscript~``$+$'' denotes the
positive part of a function, i.e.~$f^+(x):=\max\{f(x),0\}$.}
in Lemma~\ref{p2P} at the end of this paper,
we always have that $$ 2\in 
{\mathcal{P}}(\psi,\Omega).$$
In particular, the main result of this paper (i.e., the
forthcoming Theorem~\ref{1})
always holds for~$p=2$ without any further restriction.
We think that it is an interesting open problem to 
decide whether or not other
values of~$p$ belong to~${\mathcal{P}}(\psi,\Omega)$,
in general, or at least when the right hand side
of~\eqref{TX} is close to zero (e.g., 
when the obstacle is almost flat).
For instance, the property in Definition~\ref{D}
would be satisfied in presence of a
H\"older regularity theory for the horizontal
gradient for solutions of quasilinear
equations in the Heisenberg group -- namely,
if one knew that bounded solutions
of~$\div_{\H^n}\, ((\varepsilon+|\nabla_{\H^n}u|^2)^{(p/2)-1}
\nabla_{\H^n}u)=f$, with~$f$ bounded, have H\"older continuous
horizontal gradient, with interior estimates
(this would be the Heisenberg counterpart of classical
regularity results for
the Euclidean case, see, e.g.,
Theorem~1 in~\cite{Tolk}; notice also that
it would be a regularity theory for the equation,
not for the obstacle problem).
As far as we know, such a theory has not been
developed yet, not even for minimal solutions
(see, however,~\cite{Capogna, Zhong1, Zhong2}
for the case of homogeneous
equations).

The result we prove here is:

\begin{theorem}\label{1}
If~$p\in {\mathcal{P}}(\psi,\Omega)$ then
\begin{equation}\label{Tr}
\begin{split}
&0\le
\div_{\H^n}\, \Big( (\varepsilon+|\nabla_{\H^n}\bar u_\varepsilon|^2)^{(p/2)-1}
\nabla_{\H^n}\bar u_\varepsilon\Big)
\\ &\qquad\le \left(
\div_{\H^n}\, \Big( (\varepsilon+|\nabla_{\H^n}\psi|^2)^{(p/2)-1} 
\nabla_{\H^n}\psi\Big)
\right)^+.\end{split}\end{equation}
\end{theorem}

The result in Theorem~\ref{1} is quite intuitive:
when~$\bar u_\varepsilon$ does not touch the obstacle, it is free
to make the operator vanish. When it touches and sticks to it,
the operator is driven by the one of the obstacle -- and
on these touching points the obstacle has to
bend in a somewhat convex fashion, which justifies the
first inequality in~\eqref{Tr} and
superscript~``$+$'' in the right hand side of~\eqref{Tr}. 

Figure~1, in which
the thick curve represents the touching between~$\bar u_\varepsilon$
and the obstacle,
tries to describe this phenomena.
On the other hand, the set in which~$\bar u_\varepsilon$ touches
the obstacle may be very wild, so the actual proof
of Theorem~\ref{1} needs to me more technical than this.

In fact, the first inequality of~\eqref{Tr} is quite
obvious since it follows, for instance, by taking~$v:=
\bar u_\varepsilon-\varphi$
in~\eqref{ines}, with an arbitrary~$\varphi\in
C^\infty_0(\Omega,[0,+\infty))$), so the
core of~\eqref{Tr} lies on the second
inequality: nevertheless, we think it
is useful to write~\eqref{Tr} in this way to emphasize
a control from both the sides of the operator applied
to the solution.

We remark that the right hand side of~\eqref{Tr}
is always finite when~$\varepsilon>0$, and when~$\varepsilon=0$
and~$p\ge2$. In this case,~\eqref{Tr}
is an~$L^\infty$-bound and may be seen as
a regularity result for the solution
of the obstacle problem. It is worth
noticing that such regularity result holds
for~$\varepsilon=0$ as well, only
assuming that~$p\in 
{\mathcal{P}}(\psi,\Omega)$,
which is a requirement on the problem when~$\varepsilon>0$.

On the other hand, if~$\varepsilon=0$ and~$p\in(1,2)$,
the right hand side of~\eqref{Tr}
may become infinite (in this 
case~\eqref{Tr}
is true, but meaningless, stating that something is
less than or equal to an infinite quantity).

\begin{figure}[htbp]
\begin{center}
\resizebox{10cm}{!}{\input{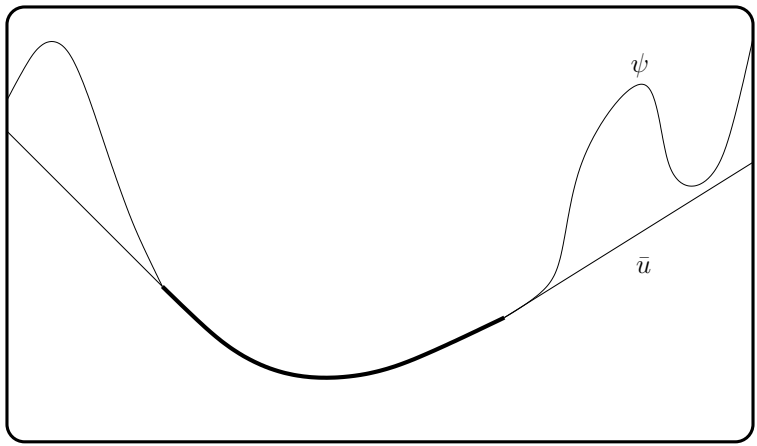}}
{\caption{\it Touching the obstacle}}
\label{FIG 1}
\end{center}
\end{figure}

In the Euclidean setting, the analogue
of~\eqref{Tr} was first obtained in~\cite{LewySt}
for the Laplacian case, and it is therefore often referred
to with the name of
Lewy-Stampacchia Estimate. It is also called Dual
Estimate, for it is, in a sense, obtained
by the duality expressed by the variational inequality~\eqref{ines}.
Other Authors refer to it
with the name of
Penalization Method, for the role played by~$\varepsilon$.

After~\cite{LewySt}, estimates of these type
became very popular and underwent
many important extensions
and strengthenings: see, among the others,~\cite{MoscoTro,
MoscoFre, DonatiF82, BocGallouet90, Mok}.
As far as we know, the estimate we prove is new in
the Heisenberg group setting, even for~$p=2$.

Hereafter, we deal with the proof of 
Theorem~\ref{1}.
First, in~\S~\ref{S1}, we prove 
Theorem~\ref{1} when~$\varepsilon>0$.

The proof when~$\varepsilon=0$ is contained
in~\S~\ref{S3} and it is based on a limit argument,
i.e., we consider the problem with~$\varepsilon>0$,
we use Theorem~\ref{1} in such a case, and then we
pass~$\varepsilon
\searrow0$. This procedure is quite delicate though,
because, as far as we know, it is not clear whether
or not the
Heisenberg group setting allows a complete H\"older regularity
theory for first derivatives (see~\cite{Dom}).
To get around this point, in~\S~\ref{S2},
we study the $L^p$-convergence of the solution~$\bar u_\varepsilon$
of the
$\varepsilon$-problem
to the solution~${\bar{u}_0}$ of the problem with~$\varepsilon=0$,
which, we believe, is of independent interest
(see, in particular Propositions~\ref{Lemma 2.3}
and~\ref{Lemma 1.23 p+}).

The paper ends with an Appendix that collects some
ancillary results needed in~\S~\ref{S2}.

\section{Proof of Theorem~\ref{1} when $\varepsilon>0$}\label{S1}

We prove~\eqref{Tr} in the simpler case~$\varepsilon>0$
(the case~$\varepsilon=0$ will be dealt with in~\S~\ref{S3}).
The technique used in this proof is a variation
of a classical penalized test function method
(see, e.g.,~\cite{MoscoTro, MoscoFre, DonatiF82, BocGallouet90,
Mok}
and references therein), and several steps
of this proof are inspired by
some estimates obtained by~\cite{Challal} in the Euclidean
case.

First of all, we set
$$ \mu := -1+\min \big\{ \inf_{\overline\Omega}\psi,
\,\inf_{\overline\Omega} u_\star\big\}\in\R$$
and we observe that
\begin{equation}\label{L1}
\bar u_\varepsilon \ge \mu
\end{equation}
a.e. in $\Omega$.
Indeed, let~$w:=\max\{ \bar u_\varepsilon, \mu\}$.
Since~$\psi$ and~$u_\star$ are below~$\mu$ in~$\Omega$,
we have that~$w\in {\mathcal{K}}$, thus
$$ 0\le {\mathcal{F}}_\varepsilon (w;\Omega)-
{\mathcal{F}}_\varepsilon (\bar u_\varepsilon;\Omega)
=-\int_{\Omega\cap \{ \bar u_\varepsilon<\mu\}
}(\varepsilon+|\nabla_{\H^n}\bar u_\varepsilon|^2)^{p/2}\le 0
,$$
and, from this,~\eqref{L1} plainly follows.

Now, let $\eta\in(0,1)$, to be taken arbitrarily small
in the sequel. Let also
\begin{equation}\label{hh}
h:=\left(
\div_{\H^n}\, \Big( (\varepsilon+|\nabla_{\H^n}\psi|^2)^{(p/2)-1} 
\nabla_{\H^n}\psi\Big)\right)^+.
\end{equation}
Notice that
\begin{equation}\label{09890}
\| h\|_{L^\infty(\Omega)}<+\infty
,\end{equation}
because~$\varepsilon>0$.
For any $t\in \R$, we consider the truncation function
$$ H_\eta(t):=\begin{cases}
0 & {\mbox{ if $t\le 0$,}} \\
t/\eta & {\mbox{ if $0< t<\eta$,}} \\
1& {\mbox{ if $t\ge\eta$.}}
\end{cases}$$
Now, we take $u_\eta$ to be a weak solution
of
\begin{equation}\label{2..8} \begin{cases}
\div_{\H^n}\, \Big( (\varepsilon+|\nabla_{\H^n}
u_\eta|^2)^{(p/2)-1} \nabla_{\H^n}u_\eta\Big) = 
h \cdot \big(1-H_\eta(\psi-u_\eta)\big) & {\mbox{ in $\Omega$,}}\\
u_\eta=\bar u_\varepsilon & {\mbox{ on $\partial \Omega$.}}
\end{cases}\end{equation}
where, as usual, the boundary datum is attained in the trace sense:
such a function $u_\eta$ may be obtained by the direct
method in the calculus of variations, by minimizing
the functional
$$ \int_\Omega\frac{1}{p}
(\varepsilon+|\nabla_{\H^n}
u(\xi)|^2)^{p/2}+F_\eta(u(\xi),\xi)\,d\xi$$
over~$u\in W^{1,p}_{\H^n}(\Omega)$, $u-\bar u_\varepsilon\in
W^{1,p}_{\H^n,0}(\Omega)$,
where
$$ F_\eta(r,\xi):=\int_0^r h(\xi) \cdot \big(1-H_\eta(\psi(\xi)-
\sigma)\big)
\,d\sigma.$$
Now, we claim that
\begin{equation}\label{A01}
u_\eta\le\psi {\mbox{ a.e. in $\Omega$}}.
\end{equation}
To establish this, we use the test function $(u_\eta-\psi)^+$
in~\eqref{2..8}.
Since, on $\partial \Omega$, we have $(u_\eta-\psi)^+=
(\bar u_\varepsilon-\psi)^+=0$, we obtain that
\begin{eqnarray*}
&& -\int_{\Omega}
\Big( (\varepsilon+|\nabla_{\H^n}
u_\eta|^2)^{(p/2)-1} \nabla_{\H^n}u_\eta\Big) \cdot \nabla_{\H^n}
(u_\eta-\psi)^+
\\ &=& \int_{\Omega}
h \cdot \big(1-H_\eta(\psi-u_\eta)\big)
(u_\eta-\psi)^+
= \int_{\Omega}
h \cdot (u_\eta-\psi)^+.
\end{eqnarray*}
Consequently, by~\eqref{hh},
\begin{eqnarray*}
&& \int_{\Omega}
\left[ \Big( (\varepsilon+|\nabla_{\H^n}
u_\eta|^2)^{(p/2)-1} \nabla_{\H^n}u_\eta\Big) \right.
\\ && \qquad \left.-
\Big( (\varepsilon+|\nabla_{\H^n}
\psi|^2)^{(p/2)-1} \nabla_{\H^n}\psi\Big)
\right]\cdot \nabla_{\H^n}
(u_\eta-\psi)^+
\\ &=&
\int_{\Omega} \left[ 
\div_{\H^n}\, \Big( (\varepsilon+|\nabla_{\H^n}
\psi|^2)^{(p/2)-1} \nabla_{\H^n}\psi\Big) -h
\right]\cdot (u_\eta-\psi)^+
\\ &\le& 0.
\end{eqnarray*}
By the strict monotonicity of the operator (i.e., by the strict convexity
of the function~$\R^{2n}\ni \zeta\mapsto
(\varepsilon+|\zeta|^2)^{p/2}$), 
it follows that
$(u_\eta-\psi)^+$ vanishes almost everywhere in $\Omega$,
proving \eqref{A01}.

Now, we claim that
\begin{equation}\label{A02}
\bar u_\varepsilon\ge u_\eta
{\mbox{ a.e. in $\Omega$}}.
\end{equation}
To verify this, we consider the test function
$$ \tau:= \bar u_\varepsilon+(u_\eta-\bar u_\varepsilon)^+ .$$
We notice that 
$$\tau=\begin{cases}
\bar u_\varepsilon & {\mbox{ in }} \{ u_\eta\le \bar u_\varepsilon\}, \\
u_\eta & {\mbox{ in }} \{ u_\eta> \bar u_\varepsilon\},
\end{cases}$$
hence $\tau
\le \psi$, due to \eqref{A01}.
Furthermore, on $\partial \Omega$, we have that 
$\tau=
\bar u_\varepsilon$, due to the boundary datum
in~\eqref{2..8}. Therefore
the obstacle problem variational inequality~\eqref{ines}
gives that
\begin{equation}\label{2.10}
\begin{split}
& 0 \le 
\int_{\Omega}
\Big( (\varepsilon+|\nabla_{\H^n}
\bar u_\varepsilon|^2)^{(p/2)-1} \nabla_{\H^n}
\bar u_\varepsilon\Big) \cdot \nabla_{\H^n}
(\tau-\bar u_\varepsilon)
\\ &\qquad= 
\int_{\Omega}
\Big( (\varepsilon+|\nabla_{\H^n}
\bar u_\varepsilon
|^2)^{(p/2)-1} \nabla_{\H^n}\bar u_\varepsilon\Big) \cdot \nabla_{\H^n}
(u_\eta-\bar u_\varepsilon)^+.
\end{split}\end{equation}
On the other hand, from~\eqref{2..8},
\begin{equation}\label{2.11}
\begin{split}
& \int_{\Omega} \Big( (\varepsilon+|\nabla_{\H^n}
u_\eta|^2)^{(p/2)-1} \nabla_{\H^n}u_\eta\Big)\cdot
\nabla_{\H^n} (u_\eta-\bar u_\varepsilon)^+
\\ &\qquad= -\int_{\Omega}h\cdot
\big(1-H_\eta(\psi-u_\eta)\big)\cdot (u_\eta-\bar u_\varepsilon)^+
\le 0.
\end{split}
\end{equation}
By \eqref{2.10} and \eqref{2.11}, we obtain that
\begin{eqnarray*}
&& \int_{\Omega} \left[
\Big( (\varepsilon+|\nabla_{\H^n}
u_\eta|^2)^{(p/2)-1} \nabla_{\H^n}u_\eta\Big) \right. \\ &&\qquad
\left. - \Big( (\varepsilon+|\nabla_{\H^n}
\bar u_\varepsilon|^2)^{(p/2)-1} \nabla_{\H^n}\bar u_\varepsilon\Big)
\right]\cdot \nabla_{\H^n}
(u_\eta-\bar u_\varepsilon)^+
\le 0.
\end{eqnarray*}
This and the strict monotonicity of the operator implies that
$(u_\eta-\bar u_\varepsilon)^+$ vanishes almost everywhere in $\Omega$,
hence proving \eqref{A02}.

Now, we claim that
\begin{equation}
\label{A03}
\bar u_\varepsilon\le u_\eta+\eta {\mbox{ in $\Omega$}}.
\end{equation}
To do this, we set
$$ \theta:= \bar u_\varepsilon -(\bar u_\varepsilon-u_\eta-\eta)^+.$$
Notice that $\theta\le\bar u_\varepsilon\le \psi$, and also that,
on $\partial \Omega$, $\theta =\bar u_\varepsilon$. As a consequence,~\eqref{ines}
gives that
\begin{equation}\label{2.12}
\begin{split}
& 0 \le
\int_{\Omega}
\Big( (\varepsilon+|\nabla_{\H^n}
\bar u_\varepsilon|^2)^{(p/2)-1} \nabla_{\H^n}
\bar u_\varepsilon\Big) \cdot \nabla_{\H^n}
(\theta-\bar u_\varepsilon)
\\ &\qquad=
-\int_{\Omega}
\Big( (\varepsilon+|\nabla_{\H^n}
\bar u_\varepsilon|^2)^{(p/2)-1} \nabla_{\H^n}
\bar u_\varepsilon\Big) \cdot \nabla_{\H^n}
(\bar u_\varepsilon-u_\eta-\eta)^+.
\end{split}\end{equation}
On the other hand,~$(\bar u_\varepsilon-u_\eta-\eta)^+=0$
on~$\partial\Omega$, and
\begin{eqnarray*} \{ \bar u_\varepsilon-u_\eta -\eta >0\}
&\subseteq& \{ \psi-u_\eta >\eta\}\\
&\subseteq& \big\{ 
1-H_\eta(\psi-u_\eta)=0\big\},\end{eqnarray*}
and therefore, by~\eqref{2..8},
\begin{equation}\label{2.13}
\begin{split}
& \int_{\Omega}
\Big( (\varepsilon+|\nabla_{\H^n}
(u_\eta+\eta)|^2)^{(p/2)-1} \nabla_{\H^n}
(u_\eta+\eta) \Big) \cdot \nabla_{\H^n}
(\bar u_\varepsilon-u_\eta-\eta)^+ 
\\=& \int_{\Omega}
\Big( (\varepsilon+|\nabla_{\H^n}
u_\eta|^2)^{(p/2)-1} \nabla_{\H^n}
u_\eta \Big) \cdot \nabla_{\H^n}
(\bar u_\varepsilon-u_\eta -\eta)^+
\\ = &
-\int_{\Omega}
h \cdot \big(1-H_\eta(\psi-u_\eta)\big)
\cdot (\bar u_\varepsilon-u_\eta -\eta)^+ =0.
\end{split}\end{equation}
Then, \eqref{2.12} and \eqref{2.13} yield that
\begin{eqnarray*}
&& \int_{\Omega} \left[
\Big( (\varepsilon+|\nabla_{\H^n}
\bar u_\varepsilon|^2)^{(p/2)-1} \nabla_{\H^n}
\bar u_\varepsilon\Big) \right.
\\ &&\qquad-\left.
\Big( (\varepsilon+|\nabla_{\H^n}
(u_\eta+\eta)|^2)^{(p/2)-1} \nabla_{\H^n}
(u_\eta+\eta) \Big)
\right]\cdot \nabla_{\H^n}
(\bar u_\varepsilon-u_\eta -\eta)^+ \\ &&\qquad\qquad\le 0.
\end{eqnarray*}
Thus, in this case, the strict monotonicity of the operator
implies that $(\bar u_\varepsilon-u_\eta-\eta)^+$ vanishes
almost everywhere in $\Omega$, and so \eqref{A03}
is established.

In particular, by~\eqref{A01},
\eqref{A03} and~\eqref{L1},
\begin{equation}\label{L2} \|u_\eta\|_{L^\infty(\Omega)}
\le 2+\|\psi\|_{L^\infty(\Omega)}
+ \|u_\star\|_{L^\infty(\Omega)}.\end{equation}
Moreover,
by \eqref{A02} and \eqref{A03},
we have that
\begin{equation}\label{781}{\mbox{$u_\eta$
converges uniformly in $\Omega$
to $\bar u_\varepsilon$}}\end{equation} as $\eta\searrow0$.

Furthermore
\begin{equation*}0 \le
\partial_r F_\eta(r,\xi)\le
h(\xi)=\left(
\div_{\H^n}\, \Big( (\varepsilon+|\nabla_{\H^n}\psi|^2)^{(p/2)-1}
\nabla_{\H^n} \psi\Big)\right)^+
\end{equation*}
hence~\eqref{Tr} follows\footnote{It
is worth pointing out that this
is the only place in the paper where we use
the condition that~$p\in {\mathcal{P}}(\psi,\Omega)$.
In particular, all the estimates in~\S~\ref{S2}
are valid for all~$p\in(1,+\infty)$.}
from~\eqref{781}
and the fact that~$p\in{\mathcal{P}}(\psi,\Omega)$
(recall~\eqref{TX} in Definition~\ref{D}).~\qed

\section{Estimating the $L^p$-distance between
$\nabla_{\H^n}{\bar{u}_0}$ and 
$\nabla_{\H^n}\bar u_\varepsilon$}\label{S2}

The purpose of this section
is to 
consider
the solution~$\bar u_\varepsilon$
of the
$\varepsilon$-problem
and the solution~${\bar{u}_0}$ of the problem with~$\varepsilon=0$,
and to bound the~$L^p$-norm of $
|\nabla_{\H^n} {\bar{u}_0}-\nabla_{\H^n}\bar u_\varepsilon|$.
Such estimate is quite technical and it is different according
to the cases $p\in(1,2]$
and $p\in[2,+\infty)$: see the forthcoming Propositions~\ref{Lemma 2.3}
and~\ref{Lemma 1.23 p+}.

As a matter of fact, we think that the estimates proved
in Propositions~\ref{Lemma 2.3}
and~\ref{Lemma 1.23 p+} are of independent interest, since
they also allow to get around the more difficult (and in general
not available in the Heisenberg group)
H\"older-type estimates.

We recall the standard notation of balls in
the Heisenberg group (in fact, we deal with the so called
Folland-Kor\'any balls, but the
Carnot-Carath\'eodory balls would be good for our purposes too).
For all~$\xi:=(z,t)\in\R^{2n}\times\R$, we define
$$\|\xi\|_{\H^n}:=\sqrt[4]{|z|^4+t^2}.$$
Then, for any~$r>0$, we set
$$ \B_r:= \big\{ \xi\in \R^{2n+1} {\mbox{ s.t. }}
\|\xi\|_{\H^n}<r\big\}.$$
We denote by~${\mathcal{L}}$
the~$(2n+1)$-dimensional
Lebesgue measure, and we observe that~${\mathcal{L}}(\B_r)$
equals, up to a multiplicative constant, to~$r^Q$, 
where~$Q:=2(n+1)$ is the homogeneous dimension of $\H^n$.
Also, for all~$g\in L^1(\B_r)$, we define the average of~$g$
in~$\B_r$ as
$$ (g)_r:= \frac{1}{ {\mathcal{L}} (\B_r)} \int_{\B_r} g.$$
In what follows, we focus on~$L^p$-estimates
around a fixed point, say~$\xi_\star$, of~$\Omega$.
Without loss of generality, we take~$\xi_\star$ to be the
origin, and we fix~$R\in(0,1)$ so small that~$\B_R
\Subset\Omega$.

Then, we denote by~${\bar{u}_0}:\Omega\rightarrow\R$ the minimizer
of problem~\eqref{PB} with~$\varepsilon=0$.
Then, for a fixed~$\varepsilon>0$, we take~$\bar 
u_\varepsilon:\B_R\rightarrow\R$
to be the minimizer of~${\mathcal{F}}_\varepsilon (u;\B_R)$
among all the functions~$u
\in W^{1,p}_{\H^n}(\B_R)$, $u\le \psi$, and~$u-{\bar{u}_0}
\in W^{1,p}_{\H^n,0}(\B_R)$. We can then extend~$\bar u_\varepsilon$
also on~$\Omega\setminus\B_R$ by setting it equal to~${\bar{u}_0}$
in such a set.
By construction
\begin{equation}\label{ar}
\begin{split} 
&\int_{\B_R} |\nabla_{\H^n} {\bar{u}_0}|^p
={\mathcal{F}}_0( {\bar{u}_0};\Omega)-
\int_{\Omega\setminus\B_R} |\nabla_{\H^n} {\bar{u}_0}|^p\\ &
\quad\le {\mathcal{F}}_0( \bar u_\varepsilon ;\Omega)-
\int_{\Omega\setminus\B_R} |\nabla_{\H^n} {\bar{u}_0}|^p
=\int_{\B_R} |\nabla_{\H^n} \bar u_\varepsilon|^p
\end{split}
\end{equation}
and
\begin{equation}\label{po}
\begin{split}&
\int_{\B_R} (\varepsilon+ |\nabla_{\H^n}\bar u_\varepsilon|^2)^{p/2}
= {\mathcal{F}}_\varepsilon (\bar u_\varepsilon;\B_R)
\\&\qquad
\le {\mathcal{F}}_\varepsilon ( {\bar{u}_0};\B_R)=
\int_{\B_R} (\varepsilon+ |\nabla_{\H^n} {\bar{u}_0}|^2)^{p/2}
.\end{split}\end{equation} 

\begin{proposition}\label{Lemma 2.3}
Assume that
\begin{equation}\label{p}
p\in(1,2].
\end{equation}
Then,
there exists~$C>0$, only depending on~$n$ and~$p$, such that
\begin{equation}\label{15bis}
\int_{\B_R} |\nabla_{\H^n} {\bar{u}_0}-
\nabla_{\H^n}\bar u_\varepsilon|^p \le C\big( 1+ (|\nabla_{\H^n}
{\bar{u}_0}|^p)_R \big)^{1-(p/2)}\varepsilon^{(p/2)^2}
R^Q.\end{equation}
\end{proposition}

\begin{proof} The technique for this proof
is inspired by the one of Lemma~2.3
of~\cite{Mu}, where a similar result was obtained
in the quasilinear Euclidean case (however, our
proof is self-contained).
We have
\begin{equation}\label{po2}
\begin{split}
& |\nabla_{\H^n}\bar u_\varepsilon-\nabla_{\H^n} {\bar{u}_0}|^2\le
\big( |\nabla_{\H^n}\bar u_\varepsilon|+|\nabla_{\H^n} {\bar{u}_0}|\big)^2
\\ &\qquad\le C \big( |\nabla_{\H^n}\bar u_\varepsilon|^2 +|\nabla_{\H^n} {\bar{u}_0}|^2\big)
.\end{split}\end{equation} Here,~$C$ is a positive constant, which
is free to be different from line to line.
By~\eqref{p},
\eqref{po} and~\eqref{po2}, we obtain
\begin{equation}\label{po3}
\begin{split}
& \int_{\B_R} (\varepsilon+ |\nabla_{\H^n} {\bar{u}_0}|^2+
|\nabla_{\H^n}\bar u_\varepsilon|^2)^{(p/2)-1}|\nabla_{\H^n}\bar u_\varepsilon-\nabla_{\H^n} {\bar{u}_0}|^2
\\ \le \,&
C \int_{\B_R} \frac{
|\nabla_{\H^n}\bar u_\varepsilon|^2 +|\nabla_{\H^n} {\bar{u}_0}|^2
}{(\varepsilon+ |\nabla_{\H^n} {\bar{u}_0}|^2+
|\nabla_{\H^n}\bar u_\varepsilon|^2)^{1-(p/2)}}
\\ =\,&
C \left( \int_{\B_R} \frac{   
|\nabla_{\H^n}\bar u_\varepsilon|^2 
}{(\varepsilon+ |\nabla_{\H^n} {\bar{u}_0}|^2+
|\nabla_{\H^n}\bar u_\varepsilon|^2)^{1-(p/2)}}
\right.\\ &\qquad+\left.
\int_{\B_R} \frac{   
|\nabla_{\H^n} {\bar{u}_0}|^2
}{(\varepsilon+ |\nabla_{\H^n} {\bar{u}_0}|^2+
|\nabla_{\H^n}\bar u_\varepsilon|^2)^{1-(p/2)}}\right)
\\ \le\, &
C \left( \int_{\B_R} \frac{
|\nabla_{\H^n}\bar u_\varepsilon|^2
}{(\varepsilon+
|\nabla_{\H^n}\bar u_\varepsilon|^2)^{1-(p/2)}}
+
\int_{\B_R} \frac{
|\nabla_{\H^n} {\bar{u}_0}|^2
}{(\varepsilon+ |\nabla_{\H^n} {\bar{u}_0}|^2)^{1-(p/2)}}\right)
\\ \le\,&
C\left(
\int_{\B_R} (\varepsilon+|\nabla_{\H^n}\bar u_\varepsilon|^2)^{p/2}
+
\int_{\B_R} (\varepsilon+|\nabla_{\H^n} {\bar{u}_0}|^2)^{p/2}\right)
\\ \le\,& C
\int_{\B_R} (\varepsilon+|\nabla_{\H^n} {\bar{u}_0}|^2)^{p/2}.
\end{split}
\end{equation}
Thus, \eqref{po3} and Lemma~\ref{from 10 to 11},
applied here with~$a:=\nabla_{\H^n}{\bar{u}_0}$ and~$b:=\nabla_{\H^n}\bar u_\varepsilon$,
yield that
\begin{equation}\label{11}
\begin{split}
& \int_{\B_R} (\varepsilon+ |\nabla_{\H^n} {\bar{u}_0}|^2+
|\nabla_{\H^n}\bar u_\varepsilon|^2)^{p/2}
\\ &\qquad\le
C\int_{\B_R}
(\varepsilon+ |\nabla_{\H^n} {\bar{u}_0}|^2+
|\nabla_{\H^n}\bar u_\varepsilon|^2)^{(p/2)-1}
|\nabla_{\H^n}\bar u_\varepsilon-\nabla_{\H^n}{\bar{u}_0}|^2+
\\ &\qquad\qquad+
C\int_{\B_R}(\varepsilon+ |\nabla_{\H^n} {\bar{u}_0}|^2)^{(p/2)}
\\ &\qquad\le
C \int_{\B_R} (\varepsilon+ |\nabla_{\H^n} {\bar{u}_0}|^2)^{(p/2)}.
\end{split}
\end{equation}
Now, from~\eqref{ar},
\begin{equation}\label{5}
\begin{split}
& \int_{\B_R} (\varepsilon+ |\nabla_{\H^n} {\bar{u}_0}|^2)^{(p/2)}
- \int_{\B_R} (\varepsilon+ |\nabla_{\H^n} \bar u_\varepsilon|^2)^{(p/2)}
\\ &\qquad\le
\int_{\B_R} (\varepsilon+ |\nabla_{\H^n} {\bar{u}_0}|^2)^{(p/2)}
- \int_{\B_R} |\nabla_{\H^n} \bar u_\varepsilon|^p
\\ &\qquad\le
\int_{\B_R} (\varepsilon+ |\nabla_{\H^n} {\bar{u}_0}|^2)^{(p/2)}
- \int_{\B_R} |\nabla_{\H^n} {\bar{u}_0}|^p.
\end{split}\end{equation}
Moreover, using~\eqref{p} and some elementary calculus,
we see that
$$ |(1+\tau)^{p/2}-\tau^{p/2}|\le C$$
for any~$\tau\ge 0$. Therefore, taking~$\tau:=\theta/\varepsilon$,
we obtain that
\begin{equation}\label{7bis0}
|(\varepsilon+\theta)^{p/2}-\theta^{p/2}|\le C\varepsilon^{p/2}
\end{equation}
for any~$\theta\ge 0$. Thus, using~\eqref{5} and~\eqref{7bis0}
with~$\theta:= |\nabla_{\H^n}{\bar{u}_0}|^2$,
we conclude that
\begin{equation}\label{6}
\int_{\B_R} (\varepsilon+ |\nabla_{\H^n} {\bar{u}_0}|^2)^{(p/2)}
- \int_{\B_R} (\varepsilon+ |\nabla_{\H^n} \bar u_\varepsilon|^2)^{(p/2)}
\le C \varepsilon^{p/2} R^Q.\end{equation}
Now, we estimate the left hand side of~\eqref{6} from below.
For this scope, we define
\begin{eqnarray*}
&& h := t \nabla_{\H^n}{\bar{u}_0}+(1-t)\nabla_{\H^n}\bar u_\varepsilon,\\
&& J:=
p\int_{\B_R} (\varepsilon+ |\nabla_{\H^n} \bar u_\varepsilon|^2)^{(p/2)-1}
\nabla_{\H^n} \bar u_\varepsilon\cdot (\nabla_{\H^n} {\bar{u}_0}-\nabla_{\H^n} \bar u_\varepsilon)
\\ {\mbox{and }} &&
\tilde J :=p \int_{\B_R}\Big[
\int_0^1(1-t)\frac{d}{dt}\Big( (\varepsilon+|h|^2)^{(p/2)-1}
h\cdot ( \nabla_{\H^n}{\bar{u}_0}-\nabla_{\H^n}\bar u_\varepsilon)
\Big)\,dt\Big].
\end{eqnarray*}
We observe that the variational inequality in~\eqref{ines}
for~$\bar u_\varepsilon$ gives that
\begin{equation}\label{J9}
J\ge 0.\end{equation}
Also, using
the Fundamental Theorem of Calculus, we
obtain
\begin{equation*}
\begin{split}
& \int_{\B_R} (\varepsilon+ |\nabla_{\H^n} {\bar{u}_0}|^2)^{(p/2)}
- \int_{\B_R} (\varepsilon+ |\nabla_{\H^n} \bar u_\varepsilon|^2)^{(p/2)}
\\ &\qquad= \int_{\B_R} \Big[ \int_0^1
\frac{d}{dt} (\varepsilon+ |t\nabla_{\H^n} {\bar{u}_0}+(1-t) \nabla_{\H^n}
\bar u_\varepsilon|^2)^{(p/2)}\,dt\Big]
\\ & \qquad= p \int_{\B_R} \Big[ \int_0^1
(\varepsilon+ |t\nabla_{\H^n} {\bar{u}_0}+(1-t) \nabla_{\H^n}
\bar u_\varepsilon|^2)^{(p/2)-1}
\\ &\qquad\qquad\times (t\nabla_{\H^n} {\bar{u}_0}+(1-t) \nabla_{\H^n}
\bar u_\varepsilon)\cdot(\nabla_{\H^n} {\bar{u}_0}-\nabla_{\H^n}
\bar u_\varepsilon)\,dt\Big]
\\ & \qquad=
p\int_{\B_R} \Big[ \int_0^1 
(\varepsilon+ |h|^2)^{(p/2)-1}
h\cdot(\nabla_{\H^n} {\bar{u}_0}-\nabla_{\H^n}
\bar u_\varepsilon)\,dt\Big] .
\end{split}
\end{equation*}
Integrating by parts the
latter integral
in~$t$ (by writing~$dt=\frac{d}{dt} (t-1)\,dt$),
and exploiting~\eqref{J9}, 
we obtain
\begin{equation}\label{7}
\begin{split}
& \int_{\B_R} (\varepsilon+ |\nabla_{\H^n} {\bar{u}_0}|^2)^{(p/2)}
- \int_{\B_R} (\varepsilon+ |\nabla_{\H^n} \bar u_\varepsilon|^2)^{(p/2)}
\\ &\qquad= J+\tilde J\ge \tilde J.
\end{split}
\end{equation}
Making use of Lemma~\ref{sim-h} -- applied here 
with~$a:=\nabla_{\H^n}{\bar{u}_0}$
and~$b:= \nabla_{\H^n}\bar u_\varepsilon$ --
we have that
\begin{equation*}
\tilde J\ge \frac1C\int_{\B_R}\Big[\int_0^1
(1-t) (\varepsilon+|t\nabla_{\H^n}{\bar{u}_0}+(1-t)\nabla_{\H^n} \bar u_\varepsilon|^2)^{(p/2)-1}
|\nabla_{\H^n}{\bar{u}_0}-\nabla_{\H^n}\bar u_\varepsilon|^2
\,dt\Big].
\end{equation*}
{F}rom this and Lemma~\ref{AT679} -- applied
here with~$\kappa:=1$
and~$\Psi(x):=x^{1-(p/2)}$, which is nondecreasing, thanks 
to~\eqref{p} --
we deduce that
\begin{equation}\label{9}
\tilde J\ge \frac1C \int_{\B_R}
(\varepsilon+|\nabla_{\H^n}{\bar{u}_0}|^2+|\nabla_{\H^n} \bar u_\varepsilon|^2)^{(p/2)-1}
|\nabla_{\H^n}{\bar{u}_0}-\nabla_{\H^n}\bar u_\varepsilon|^2.
\end{equation}
In view of~\eqref{6}, \eqref{7}
and~\eqref{9}, we conclude that
\begin{equation}\label{10}
\int_{\B_R}
(\varepsilon+|\nabla_{\H^n}{\bar{u}_0}|^2+|\nabla_{\H^n} \bar u_\varepsilon|^2)^{(p/2)-1}
|\nabla_{\H^n}{\bar{u}_0}-\nabla_{\H^n}\bar u_\varepsilon|^2 \le C \varepsilon^{p/2} R^Q.
\end{equation}
Then,~\eqref{15bis} follows from~\eqref{11}, \eqref{10} and
Lemma~\ref{19pL}, applied here with~$f:=
\nabla_{\H^n}{\bar{u}_0}$ and~$g:=\nabla_{\H^n}\bar u_\varepsilon$.
\end{proof}

In the degenerate case $p\in[2,+\infty)$
the estimate obtained in Proposition~\ref{Lemma 2.3}
for the singular case $p\in(1,2]$
needs to be modified according to the
following result:

\begin{proposition}\label{Lemma 1.23 p+}
Suppose that \begin{equation}\label{p+}
p\in[2,+\infty).\end{equation}
Then, there exists $C>0$,
only depending on $n$ and $p$, such that
$$ \int_{\B_R} |{\nabla_{\H^n}} {\bar{u}_0}-{\nabla_{\H^n}} \bar u_\varepsilon|^p\le C
\big( 1+
(|\nabla_{\H^n}{\bar{u}_0}|^p)_R \big)^{1-(1/p)}\varepsilon
R^Q.$$
\end{proposition}

\begin{proof} The variational inequalities~\eqref{ines}
for ${\bar{u}_0}$ and $\bar u_\varepsilon$ imply that
\begin{eqnarray*}
&& \int_{\B_R} |{\nabla_{\H^n}} {\bar{u}_0}|^{p-2}{\nabla_{\H^n}} {\bar{u}_0}
\cdot ({\nabla_{\H^n}}\bar u_\varepsilon-{\nabla_{\H^n}} {\bar{u}_0})\ge 0
\\ {\mbox{and }} &&
\int_{\B_R} (\varepsilon+
|{\nabla_{\H^n}} \bar u_\varepsilon|^2)^{(p/2)-1}{\nabla_{\H^n}} \bar u_\varepsilon
\cdot ({\nabla_{\H^n}} {\bar{u}_0}-{\nabla_{\H^n}} \bar u_\varepsilon)\ge 0.
\end{eqnarray*}
Consequently,
$$ \int_{\B_R} \Big( |{\nabla_{\H^n}} {\bar{u}_0}|^{p-2}{\nabla_{\H^n}} {\bar{u}_0}
-(\varepsilon+
|{\nabla_{\H^n}} \bar u_\varepsilon|^2)^{(p/2)-1}{\nabla_{\H^n}} \bar u_\varepsilon
\Big)\cdot ({\nabla_{\H^n}} {\bar{u}_0}-{\nabla_{\H^n}} \bar u_\varepsilon)\le 0.$$
Using this and \eqref{Fu1} of Lemma~\ref{FFU}, applied here
with $A:={\nabla_{\H^n}} {\bar{u}_0}$
and $B:={\nabla_{\H^n}} \bar u_\varepsilon$, we obtain
\begin{gather*}
\int_{\B_R} |{\nabla_{\H^n}} {\bar{u}_0}-{\nabla_{\H^n}}
\bar u_\varepsilon|^p\\ \le C
\int_{\B_R} \Big( |{\nabla_{\H^n}}
{\bar{u}_0}|^{p-2}{\nabla_{\H^n}} {\bar{u}_0}
-|{\nabla_{\H^n}}\bar u_\varepsilon|^{p-2}{\nabla_{\H^n}} \bar u_\varepsilon
\Big)\cdot ({\nabla_{\H^n}} {\bar{u}_0}-{\nabla_{\H^n}} \bar u_\varepsilon)
\\ \le C
\int_{\B_R} \Big( (\varepsilon+
|{\nabla_{\H^n}} \bar u_\varepsilon|^2)^{(p/2)-1}
{\nabla_{\H^n}} \bar u_\varepsilon
-  |{\nabla_{\H^n}} \bar u_\varepsilon|^{p-2}
{\nabla_{\H^n}} \bar u_\varepsilon
\Big)\cdot ({\nabla_{\H^n}} {\bar{u}_0}-{\nabla_{\H^n}} \bar u_\varepsilon)
.\end{gather*}
This and Corollary \ref{FuC}, applied here with
$a:={\nabla_{\H^n}}\bar u_\varepsilon$, give
\begin{gather*}
\int_{\B_R} |{\nabla_{\H^n}} {\bar{u}_0}-{\nabla_{\H^n}}
\bar u_\varepsilon|^p\\ \le C
\int_{\B_R} \Big( (\varepsilon+
|{\nabla_{\H^n}} \bar u_\varepsilon|^2)^{(p/2)-1}
-  |{\nabla_{\H^n}} {\bar u_\varepsilon}|^{p-2}
\Big)\,|{\nabla_{\H^n}} \bar u_\varepsilon| \, |{\nabla_{\H^n}} {\bar{u}_0}-
{\nabla_{\H^n}} \bar u_\varepsilon|
\\ \le C \varepsilon \int_{\B_R}
(\varepsilon+|{\nabla_{\H^n}}\bar u_\varepsilon|^2)^{(p-2)/2}
\big( |{\nabla_{\H^n}} {\bar{u}_0}|+|{\nabla_{\H^n}} \bar u_\varepsilon|\big).
\end{gather*}
Therefore, recalling \eqref{p+}, noticing that
$$ \frac{p-2}p+\frac1p+\frac1p=1$$
and using the Generalized H\"older Inequality with
the three exponents $p/(p-2)$, $p$ and $p$, we obtain
\begin{gather*}
\int_{\B_R} |{\nabla_{\H^n}} {\bar{u}_0}-{\nabla_{\H^n}} \bar u_\varepsilon|^p\\ \le
C \varepsilon \left( \int_{\B_R}
(\varepsilon+|{\nabla_{\H^n}}\bar u_\varepsilon|^2)^{p/2}\right)^{(p-2)/p}
\left( \int_{\B_R}
\big( |{\nabla_{\H^n}} {\bar{u}_0}|^p+|{\nabla_{\H^n}} \bar u_\varepsilon|^p\big)\right)^{1/p}
R^{Q/p}
.\end{gather*}
Then, by the minimal property of ${\bar{u}_0}$
in~\eqref{ar},
\begin{gather*}
\int_{\B_R} |{\nabla_{\H^n}} {\bar{u}_0}-{\nabla_{\H^n}} \bar u_\varepsilon|^p\\ \le
C \varepsilon \left( \int_{\B_R}
(\varepsilon+|{\nabla_{\H^n}}\bar u_\varepsilon|^2)^{p/2}\right)^{(p-2)/p}
\left( \int_{\B_R}
|{\nabla_{\H^n}} \bar u_\varepsilon|^p \right)^{1/p}
R^{Q/p}
\\ \le
C \varepsilon \left( \int_{\B_R}
(\varepsilon+|{\nabla_{\H^n}}\bar u_\varepsilon|^2)^{p/2}\right)^{(p-1)/p}
R^{Q/p}
\\ \le
C \varepsilon \left( R^Q +\int_{\B_R} |{\nabla_{\H^n}}\bar u_\varepsilon|^p\right)^{(p-1)/p}
R^{Q/p}
\\ \le
C \varepsilon \left( R^Q +\int_{\B_R} |{\nabla_{\H^n}}{\bar{u}_0}|^p
\right)^{(p-1)/p}
R^{Q/p},\end{gather*}
from which the desired result follows.
\end{proof}

\begin{corollary}
For all~$p\in(1,+\infty)$, we have that
\begin{equation}\label{C1}
\lim_{\varepsilon\searrow0}
\| {\nabla_{\H^n}}\bar u_\varepsilon - 
{\nabla_{\H^n}}{\bar{u}_0}\|_{L^p(\B_R)}=0.
\end{equation}
Also, there exist a subsequence of~$\varepsilon$'s
and a function~$G\in L^p(\B_R)$ such 
that
\begin{equation}\label{C2}
|{\nabla_{\H^n}}\bar u_\varepsilon (x)|\le G(x)
\end{equation}
for almost every~$x\in\B_R$.

Furthermore, if we set
\begin{equation}\label{C4}
\Gamma_\varepsilon:=
(\varepsilon+|\nabla_{\H^n}\bar u_\varepsilon|^2)^{(p/2)-1}
\nabla_{\H^n}\bar u_\varepsilon,\end{equation}
then there exist a subsequence of~$\varepsilon$'s
and a function~$G_\star \in L^1(\B_R)$ 
such that
\begin{equation}\label{C3}
|\Gamma_\varepsilon(x)|\le G_\star(x)
\end{equation}
for almost every~$x\in\B_R$.
\end{corollary}

\begin{proof} We obtain~\eqref{C1} from
Propositions~\ref{Lemma 2.3}
and~\ref{Lemma 1.23 p+}, according to whether~$p\in(1,2)$ or~$p\in
[2,+\infty)$.

{F}rom~\eqref{C1}, one deduces~\eqref{C2}
(see, e.g., Theorem~4.9(b) in~\cite{Brezis}).

Now, we define~$G_\star:= 2^{(p/2)}(G+G^{p-1})$.
We observe that~$G_\star\in L^1(\B_R)$, since~$G\in L^p(\B_R)
\subseteq L^1(\B_R)$ and~$G^{p-1} \in L^{p/(p-1)}(\B_R)
\subseteq L^1(\B_R)$ . So, in order to obtain the
desired result, we have only to show that the inequality
in~\eqref{C3} holds true.

For this, we notice that, if~$p\in(1,2)$,
\begin{eqnarray*}
&& |\Gamma_\varepsilon|=
\frac{|\nabla_{\H^n}\bar u_\varepsilon|}{
(\varepsilon+|\nabla_{\H^n}\bar u_\varepsilon|^2)^{1-(p/2)} }
= \frac{|\nabla_{\H^n}\bar u_\varepsilon|^{p-1}
|\nabla_{\H^n}\bar u_\varepsilon|^{2-p}}{
(\varepsilon+|\nabla_{\H^n}\bar u_\varepsilon|^2)^{1-(p/2)} }
\\ &&\quad\le
\frac{|\nabla_{\H^n}\bar u_\varepsilon|^{p-1}
(\varepsilon+|\nabla_{\H^n}\bar u_\varepsilon|^2)^{(2-p)/2}}{
(\varepsilon+|\nabla_{\H^n}\bar u_\varepsilon|^2)^{1-(p/2)} }
=|\nabla_{\H^n}\bar u_\varepsilon|^{p-1}
\le G^{p-1},
\end{eqnarray*}
which implies~\eqref{C3} in this case.

On the other hand, if~$p\in[2,+\infty)$,
\begin{eqnarray*}
&& |\Gamma_\varepsilon|\le 2^{(p/2)-1}
\big(\varepsilon^{(p/2)-1}
+|\nabla_{\H^n}\bar u_\varepsilon|^{p-2}
\big)
\,|\nabla_{\H^n}\bar u_\varepsilon|
\\ &&\quad\le 2^{(p/2)-1}
(1+G^{p-2}) G,
\end{eqnarray*}
which implies~\eqref{C3} in this case too.
\end{proof}

\section{Proof of Theorem~\ref{1} when~$\varepsilon=0$}\label{S3}

By Theorem~\ref{1} (for~$\varepsilon>0$,
which has been proved in~\S~\ref{S1}), we know
that, for a sequence~$\varepsilon\searrow0$,
\begin{equation}\label{C5}
0\le \int_{\B_R} \Gamma_\varepsilon\cdot\nabla\varphi
\le \int_{\B_R}\left(
\div_{\H^n}\, 
\Big( (\varepsilon+|\nabla_{\H^n}\psi|^2)^{(p/2)-1}
\nabla_{\H^n}\psi\Big)
\right)^+ \varphi,
\end{equation}
for any~$\varphi\in C^\infty_0(\B_R,[0,+\infty))$, as long as~$\B_R
\subset\Omega$, where~$\Gamma_\varepsilon$ is as in~\eqref{C4}.

By possibly taking subsequences,
in the light of~\eqref{C1} and~\eqref{C3},
we have that
$$ \lim_{\varepsilon\searrow0}\Gamma_\varepsilon=
|\nabla_{\H^n}{\bar{u}_0}|^{p-2}  
\nabla_{\H^n}{\bar{u}_0}$$
almost everywhere in~$\B_R$, and that~$\Gamma_\varepsilon$
is equidominated in~$L^1(\B_R)$. Consequently,
we can pass to the limit in~\eqref{C5} via the
Dominated Convergence Theorem and obtain~\eqref{Tr}
for~${\bar{u}_0}$.
This completes the proof of Theorem~\ref{1} also
when~$\varepsilon=0$.~\qed

\section*{Appendix}

In this appendix, we collect some technical estimates of
general interest that will be used in the proofs of
the main results of this paper.

We start with some classical estimates (see, e.g. Lemma~3
in~\cite{Fuchs} and references therein), which
turns out to be quite useful to deal with nonlinear
operators of degenerate type:

\begin{lemma}\label{FFU}
Let $M\in\N$, $M\ge 1$, and $p\in[2,+\infty)$. Then, there
exists $C>1$, only depending on $M$ and $p$, such that,
for any $A$, $B\in \R^M$,
\begin{equation}\label{Fu1}
|A-B|^p \le C
\Big( |A|^{p-2} A-|B|^{p-2}B\Big)
\cdot(A-B)
\end{equation}
and
\begin{equation}\label{Fu2}
\Big|  |A|^{p-2} A-|B|^{p-2}B\Big|\le C
|A-B|\,\Big( |A|^{p-2} +|B|^{p-2} \Big)
.\end{equation}
\end{lemma}

\begin{corollary}\label{FuC}
Let $N\in\N$ and
and $p\in[2,+\infty)$. Then, there
exists $C>1$, only depending on $N$ and $p$, such that
for any $\varepsilon>0$ and any $a\in\R^N$
$$\big( (\varepsilon+|a|^2)^{(p/2)-1}-|a|^{p-2}\big) \,|a|
\le C\varepsilon  (\varepsilon+|a|^2)^{(p-2)/2}.$$
\end{corollary}

\begin{proof} We let $A:=(a,\varepsilon)$ and $B:=(a,0)\in\R^{N+1}$
and we exploit \eqref{Fu2}. We obtain
\begin{equation*}\begin{split}
& 2C\varepsilon  (\varepsilon+|a|^2)^{(p-2)/2}
\\ &\qquad\ge
C \varepsilon \Big( (\varepsilon+|a|^2)^{(p-2)/2}+|a|^{p-2}\Big)
\\ &\qquad=
C |A-B|\,\Big( |A|^{p-2} +|B|^{p-2} \Big)
\\ &\qquad \ge
\Big|  |A|^{p-2} A-|B|^{p-2}B\Big|
\\ &\qquad=
\Big| (\varepsilon+|a|^2)^{(p-2)/2}(a,\varepsilon)-|a|^{p-2}(a,0)\Big|
\\ &\qquad=
\Big|  \Big(  \big( (\varepsilon+|a|^2)^{(p-2)/2}-|a|^{p-2}\big) a,\,
(\varepsilon+|a|^2)^{(p-2)/2} \varepsilon
\Big) \Big|\\ &\qquad
\ge
\big( (\varepsilon+|a|^2)^{(p-2)/2}-|a|^{p-2}\big) \,|a|,
\end{split}
\end{equation*}
as desired.
\end{proof}

In the subsequent Lemmata~\ref{sim-h} and~\ref{from 10 to 11},
we collect some simple, but interesting,
estimates that are
used in Proposition~\ref{Lemma 2.3}:

\begin{lemma}\label{sim-h}
Let~$N\in\N$, $N\ge 1$, $t\in[0,1]$, $\varepsilon>0$, and~$a$, $b\in\R^N$.
Let~$h(t) := t a+(1-t) b$.
Then, there exists~$C>1$, only depending on~$N$
and~$p$, such that
$$ \frac{d}{dt} \Big(  (\varepsilon+|h|^2)^{(p/2)-1} h\cdot
(a-b)\Big)\ge \frac{1}{C}
(\varepsilon+|ta+(1-t)b|^2)^{(p/2)-1}
|a-b|^2.$$
\end{lemma}

\begin{proof} We have
\begin{equation*}\begin{split}
& \frac{d}{dt} \Big(  (\varepsilon+|h|^2)^{(p/2)-1} h\cdot 
(a-b)\Big)
=
\frac{d}{dt} \Big(  (\varepsilon+|h|^2)^{(p/2)-1} h \Big)\cdot
(a-b)
\\ &\qquad
= (\varepsilon+|h|^2)^{(p/2)-2} \big( \varepsilon+(p-1)|h|^2\big)
\frac{dh}{dt}\cdot (a-b)
\\ &\qquad
\ge \frac{1}{C} (\varepsilon+|h|^2)^{(p/2)-1}
|a-b|^2
\\ &\qquad=\frac{1}{C}
(\varepsilon+|ta+(1-t)b|^2)^{(p/2)-1}
|a-b|^2,\end{split}\end{equation*}
as desired.\end{proof}

\begin{lemma}\label{from 10 to 11}
Let
\begin{equation}\label{22p}
p\in(1,2].\end{equation}
Let~$N\in\N$, $N\ge 1$, $\varepsilon>0$, and~$a$, $b\in\R^N$.
Then, there exists~$C>1$, only depending on~$N$ and~$p$, such that
$$
(\varepsilon+ |a|^2+
|b|^2)^{p/2} \le C\Big[ 
(\varepsilon+ |a|^2+
|b|^2)^{(p/2)-1}
|b-a|^2
+(\varepsilon+ |a|^2)^{(p/2)} \Big].$$
\end{lemma}

\begin{proof}
We have  
$$ |b|^2=
|b -a+a|^2
\le \big(
|b-a| +|a|
\big)^2
\le C \big( |b -a|^2
+|a|^2\big) $$
and so
\begin{eqnarray*}
&& (\varepsilon+ |a|^2+
|b|^2)^{p/2} \\ &=& (\varepsilon+ |a|^2+
|b|^2)^{(p/2)-1}
(\varepsilon+ |a|^2+
|b|^2)
\\ &\le& C (\varepsilon+ |a|^2+
|b|^2)^{(p/2)-1}
(\varepsilon+ |a|^2+
|b-a|^2)
\\ &=&
 C (\varepsilon+ |a|^2+
|b|^2)^{(p/2)-1}|b-a|^2
+ C (\varepsilon+ |a|^2+
|b|^2)^{(p/2)-1}
(\varepsilon+ |a|^2).
\end{eqnarray*}
Therefore, by~\eqref{22p},
\begin{equation*}\begin{split} &
(\varepsilon+ |a|^2+
|b|^2)^{p/2} \\ &\qquad\le C
(\varepsilon+ |a|^2+
|b|^2)^{(p/2)-1}
|b-a|^2+C(\varepsilon+ |a|^2)^{(p/2)},
\end{split}\end{equation*}
that is the desired claim.\end{proof}

The following result
deals with some technical
estimates on monotone integrands.

\begin{lemma}\label{AT679}
Let~$N\in\N$, $N\ge1$. Let~$\kappa\in \{0,1\}$.
Let~$\varepsilon$, $\varepsilon'>0$. Let~$a$, $b\in 
\R^{N}$.
Let~$\Psi:[\varepsilon,+\infty)\rightarrow
[\varepsilon',+\infty)$ be a 
measurable and nondecreasing function. Then
\begin{equation}\label{DES}
\int_0^1 \frac{(1-t)^\kappa}{\Psi 
(\varepsilon+|ta+(1-t)b|^2)}\,dt
\ge \frac{1}{ 2 \Psi (\varepsilon+|a|^2+|b|^2)}.
\end{equation}
\end{lemma}

\begin{proof} If~$|a|\le |b|$, for any~$t\in[0,1]$,
\begin{eqnarray*} && |ta+(1-t)b|^2\le 
t^2 |a|^2 +(1-t)^2 |b|^2+2t(1-t)|a||b|
\\ &&\qquad \le t^2 |b|^2 +(1+t^2-2t) |b|^2+2t(1-t)|b|^2= |b|^2.
\end{eqnarray*}
On the other hand, if~$|a|\ge |b|$, for any~$t\in[0,1]$,
\begin{eqnarray*} && |ta+(1-t)b|^2\le
t^2 |a|^2 +(1-t)^2 |b|^2+2t(1-t)|a||b|
\\ &&\qquad \le t^2 |a|^2 +(1+t^2-2t) |a|^2+2t(1-t)|a|^2= |a|^2.
\end{eqnarray*}
In any case,
$$ \varepsilon+|ta+(1-t)b|^2\le \varepsilon+|a|^2+|b|^2$$
and the claim follows from the monotonicity of~$\Psi$.
\end{proof}

The next is a useful H\"older/$L^p$ type estimate,
that is exploited
in Proposition~\ref{Lemma 2.3}.

\begin{lemma}\label{19pL}
Let~$N\in\N$, $N\ge 1$. Let~$f$, $g\in L^p (\B_R,\R^N)$.
Suppose that
\begin{equation}\label{p56}
p\in(1,2].\end{equation}
Then
\begin{equation*}\begin{split}
& \int_{\B_R}|f-g|^p
\\
&\qquad\le \left(\int_{\B_R}
(\varepsilon+|f|^2+|g|^2)^{(p/2)-1}
|f-g|^2\right)^{p/2}
\\ &\qquad\qquad\times \left(\int_{\B_R}
(\varepsilon+|f|^2+|g|^2)^{p/2}
\right)^{(2-p)/2}.
\end{split}\end{equation*}
\end{lemma}

\begin{proof}
We observe that
\begin{eqnarray*}
&& |f-g|^p\\ &
=& \Big[
(\varepsilon+|f|^2+|g|^2)^{(p/2)-1}
|f-g|^2\Big]^{p/2}
\Big[
(\varepsilon+|f|^2+|g|^2)^{p/2}
\Big]^{(2-p)/2},\end{eqnarray*}
and so the desired result follows
from the H\"older Inequality
with exponents~$2/p$ and~$2/(2-p)$, which can be used here
due to~\eqref{p56}.
\end{proof}

To end this paper, we remark that
Definition~\ref{D} is always nonvoid
(independently of~$\psi$ and~$\Omega$), in the
sense that

\begin{lemma}\label{p2P}
$ 2\in {\mathcal{P}}(\psi,\Omega)$.
\end{lemma}

\begin{proof} The functional in~\eqref{905}
when~$p=2$ boils down to
\begin{equation}\label{905a} \int_\Omega
\frac{1}{2} |\nabla_{\H^n}
u(\xi)|^2+F_k(u(\xi),\xi)\,d\xi,\end{equation}
up to an additive constant that does not play any
role in the minimization.
Hence,
if~$u_k$ minimizes this functional, we have that
$$-
\int_\Omega \nabla_{\H^n} u_k(\xi)\cdot\nabla_{\H^n}\varphi(\xi)
\,d\xi
=\int_\Omega \partial_r F_k(u_k(\xi),\xi)\,\varphi(\xi)\,d\xi$$
for any~$\varphi\in C^\infty_0(\Omega)$.

Accordingly, if also~$u_k$ approaches some~$u_\infty$
uniformly in~$\Omega$, it follows that
\begin{equation}\label{905bb}
\begin{split}
& \int_\Omega u_\infty \Delta_{\H^n}\varphi=
\lim_{k\rightarrow+\infty} \int_\Omega u_k \Delta_{\H^n}\varphi
\\ &\qquad= \lim_{k\rightarrow+\infty} -\int_\Omega
\nabla_{\H^n}u_k \cdot \nabla_{\H^n}\varphi
= \lim_{k\rightarrow+\infty} \int_\Omega
\partial_r F_k(u_k,\xi)\,\varphi
\end{split}
\end{equation}
for any~$\varphi\in C^\infty_0(\Omega)$.

Also, from~\eqref{1C4},
$$ 0\le \partial_r F_k \le
\left(
\Delta_{\H^n}
\psi\right)^+$$
and so~\eqref{905bb} gives that
\begin{equation}\label{905b}
0\le \int_\Omega u_\infty \Delta_{\H^n}\varphi
\le \int_\Omega\left(
\Delta_{\H^n}
\psi\right)^+\varphi
\end{equation}
for any~$\varphi\in C^\infty_0(\Omega,[0,+\infty))$.

On the other hand, since~$u_k$ is a minimizer
for~\eqref{905a}, we have that
$$ \sup_{k\in\N} \| \nabla_{\H^n}u_k\|_{L^2(\Omega)}<+\infty$$
and so, up to a subsequence, we may suppose
that~$\nabla_{\H^n}u_k$ converges to some~$\nu\in {L^2(\Omega)}$
weakly in~${L^2(\Omega)}$. It follows from the
uniform convergence of~$u_k$ that
\begin{eqnarray*}
&& 
-\int_\Omega\nu\cdot\nabla_{\H^n}\varphi=
-\lim_{k\rightarrow+\infty}
\int_\Omega \nabla_{\H^n}u_k\cdot \nabla_{\H^n}\varphi
\\ &&\qquad=\lim_{k\rightarrow+\infty}
\int_\Omega u_k \,\Delta_{\H^n}\varphi
=\int_\Omega u_\infty\,\Delta_{\H^n}\varphi
\end{eqnarray*}
for any~$\varphi\in C^\infty_0(\Omega)$. That is,
$\nabla_{\H^n}u_\infty=\nu$ in the sense of distributions,
and so as a function. In particular,~$
\nabla_{\H^n}u_\infty\in {L^2(\Omega)}$, and therefore~\eqref{905b}
yields that
\begin{eqnarray*}
0\le
\int_\Omega \nabla_{\H^n} u_\infty \cdot
\nabla_{\H^n}\varphi
\le \int_\Omega \left(\Delta_{\H^n}
\psi\right)^+\varphi,
\end{eqnarray*}
for any~$\varphi\in C^\infty_0(\Omega,[0,+\infty))$. This shows
that~$u_\infty$ satisfies~\eqref{TX}, in the distributional sense,
hence as a function.
\end{proof}

\section*{Acknowledgmens}

We thank Luca Capogna, Juan J. Manfredi,
Xiao Zhong and
William Ziemer for some interesting observations.
EV is supported by FIRB project ``Analysis
and Beyond''.

\bigskip\bigskip

{ {\em Andrea Pinamonti}   

Universit\`a di Trento

Dipartimento di Matematica  

via Sommarive, 14 

I-38123 Povo (TN), Italy}

{\tt pinamonti@science.unitn.it}

\bigskip

{ {\em Enrico Valdinoci}   

Universit\`a di Roma Tor Vergata

Dipartimento di Matematica  

via della ricerca scientifica, 1

I-00133 Rome, Italy}

{\tt enrico@math.utexas.edu}

\end{document}